\documentclass[11pt]{amsart}

\usepackage{enumerate,graphicx}
\usepackage[latin1]{inputenc}

\DeclareMathOperator{\tr}{tr}

\DeclareMathOperator{\Id}{Id}

\DeclareMathOperator{\dbar}{\bar \partial} 
\DeclareMathOperator{\dom}{dom} 
\DeclareMathOperator{\Index}{Index} 
\DeclareMathOperator{\coker}{coker}

\renewcommand{\Im}{\operatorname{Im}}
\renewcommand{\Re}{\operatorname{Re}}

\renewcommand{\ker}{\operatorname{ker}}

\newcommand{\R}{\mathbb{R}}
\newcommand{\C}{\mathbb{C}}

\newcommand{\Z}{\mathbb{Z}}
\renewcommand{\H}{\mathbb{H}}  

\newcommand{\Dir}{\mathcal{D}}

\numberwithin{equation}{section}  

\theoremstyle{plain}
\newtheorem{Thm}{Theorem}
\newtheorem{Lem}{Lemma}
\newtheorem{Pro}[Lem]{Proposition}

\newtheorem*{The*}{Theorem}

\theoremstyle{definition}

\theoremstyle{remark}

\usepackage[
  breaklinks, 
  colorlinks=true, 
  linkcolor=black, 
  anchorcolor=black,
  citecolor=black,
  filecolor=black,
  menucolor=black,
  urlcolor=blue,
  bookmarks,
  bookmarksnumbered,
  hyperfootnotes=false,
  pdfpagelabels,
  pdfstartview=FitH, 
  pdfpagemode=UseOutlines, 
  pdfnewwindow=true,
  pagebackref=false, 
  bookmarksopen=false,
  bookmarks=true
]{hyperref}
\usepackage[figure]{hypcap}

\begin{document}
\title[Conformal deformations and elliptic boundary value
problems]{Conformal deformations of immersed discs in $\R^3$ and
  elliptic boundary value problems}

\author{Christoph Bohle}
\author{Ulrich Pinkall}

\address{Christoph Bohle\\
  Mathematisches Institut der Universit\"at T\"ubingen\\
  Auf der Morgenstelle 10\\
  72076 T\"ubingen \\
  Germany}

\address{Ulrich Pinkall\\
Technische Universit\"at Berlin\\Institut f\"ur Mathematik\\
Stra{\ss}e des 17.\ Juni 136\\
10623 Berlin\\ Germany}

\email{bohle@mathematik.uni-tuebingen.de, pinkall@math.tu-berlin.de}


\date{\today}

\begin{abstract}
  Boundary value problems for operators of Dirac type arise naturally
  in connection with the conformal geometry of surfaces immersed in
  Euclidean 3--space. Recently such boundary value problems have been
  successfully applied to a variety of problems from computer graphics.
  Here we investigate under which conditions these boundary value
  problems are elliptic and self--adjoint. We show that under certain
  periodic deformations of the boundary data our operators exhibit
  non-trivial spectral flow.
\end{abstract}

\thanks{First author supported by DFG Sfb/Tr 71 ``Geometric Partial
  Differential Equations'', second author supported by DFG Sfb/Tr 109
  ``Discretization in Geometry and Dynamics''.  Both authors
  additionally supported by the Hausdorff Institute of Mathematics in
  Bonn.}

\maketitle

\section{Introduction}
An elliptic operator on a compact manifold $M$ without boundary can be
extended to a Fredholm operator between appropriate Sobolev spaces.
If in addition the operator is formally self--adjoint, this extension
is self--adjoint. As a consequence, the spectrum is then real and
discrete and there exists a basis of eigenvectors.  On compact
manifolds with boundary the situation is essentially the same if one
imposes boundary conditions that are elliptic and self--adjoint.  This
paper investigates the conditions for ellipticity and
self--adjointness of certain local boundary value problems that arise
in surface theory and computer graphics.

The elliptic operators we deal with are operators of Dirac type that
allow to describe conformal deformations of immersed surfaces in
Euclidean 3--space. They can be most easily described within the
quaternionic approach \cite{KPP} of G.~Kamberov, F.~Pedit, and the
second author (see also \cite{PP,T2} for different perspectives on the
application of Dirac operators to surface theory).  Given an immersion
$f\colon M \rightarrow \R^3=\Im(\H)$ into Euclidean 3--space viewed as
the imaginary quaternions, every conformal deformation $\tilde f$ of
$f$ has a differential of the form
\[ d\tilde f=\bar \lambda df\lambda, \] where $\lambda \colon M
\rightarrow \H_*$ is a quaternion valued function satisfying a Dirac
type equation
\begin{equation}
  \label{eq:Dirac-rho}
  \Dir \lambda = \rho \lambda
\end{equation}
for some real valued function $\rho$ with $\Dir$ an elliptic operator
attached to $f$. By conformal deformation we mean here an immersion
$\tilde f$ that induces the same conformal structure as the original
immersion $f$ and is topologically equivalent to $f$ (i.e., in the
same regular homotopy class).

In \cite{CPS} it is shown that this relation between surface theory
and Dirac operators can be turned into an efficient method for
implementing conformal deformations in computer graphics.  The idea of
\cite{CPS} is to ``reverse'' the above correspondence that assigns to
a conformal deformation $\tilde f$ of $f$ the Dirac potential $\rho$
in \eqref{eq:Dirac-rho} and to describe conformal deformations $\tilde
f$ of a given immersion $f$ by the corresponding potentials $\rho$.
This makes sense, because by the above every conformal deformation
$\tilde f$ of a given $f$ can indeed be obtained from a suitable real
function $\rho$ by solving for $\lambda$ in the kernel of $\Dir -\rho$
and integrating $d\tilde f=\bar \lambda df\lambda$.

The main difficulty with this idea (apart from controlling the periods
of $d\tilde f=\bar \lambda df\lambda$ if $M$ has non--trivial
topology) is that the correspondence between $\rho$ and $\tilde f$ is
not a bijection. For example, given a real function $\rho$, the
operator $\Dir-\rho$ might not have a kernel; and if it has one, the
kernel might not be 1--dimensional and sections in the kernel might
have zeros, so that $\tilde f$ is not immersed. As shown in
\cite{CPS}, the first issue (that $\ker(\Dir-\rho)$ might be empty)
can be efficiently dealt with by allowing to modify $\rho$ by a real
constant $\sigma$, that is, by taking for $\lambda$ an eigenspinor
\[ (\Dir -\rho) \lambda = \sigma \lambda \] with $\sigma$ an
eigenvalue of preferably small modulus (the second issue can be
ignored for many applications in computer graphics, because
generically a non--trivial eigenspace will be 1--dimensional).  
For the method of \cite{CPS} to work it is crucial to make sure that
the spectrum of $\Dir-\rho$ is (or at least contains) a non--empty
real point spectrum.  This is always the case if the underlying
surface $M$ is compact and has no boundary. For a compact surface $M$
with boundary this is still the case if one imposes elliptic and
self--adjoint boundary conditions.

This paper derives and discusses geometric conditions for the
ellipticity and self--adjointness of local boundary conditions (i.e.,
pointwise conditions on the restriction $\lambda_{|\partial M}$ of the
spinor~$\lambda$ to the boundary) for Dirac operators induced by
immersions of compact surfaces with boundary. This kind of boundary
conditions seems to be most relevant for applications in computer
graphics.

In Section~\ref{sec:spin-trafo} we give a brief, but self--contained
review of the quaternionic approach to Dirac type operators that arise
in the context of surface theory in Euclidean 3--space.  In
Section~\ref{sec:elliptic-sa-bc} we geometrically characterize the
ellipticity and self--adjointness of local boundary conditions for
such Dirac type operators.  In Section~\ref{sec:index} we compute the
Fredholm index of elliptic local boundary conditions in the case that
the underlying surface is a disc.  Such boundary conditions turn out
to be homotopy equivalent to problems studied by Vekua in the early
days of index theory. In Section~\ref{sec:spectral-flow} we compute
the spectral flow in the case of periodic families of self--adjoint,
elliptic local boundary conditions for conformal immersions of the
disc.  In Section~\ref{sec:dirac-spheres} we conclude the paper by
describing a relation between spectral flow and the Dirac spectrum of
the round 2--sphere $S^2$.

\section{Conformal deformations and the Dirac operator $\Dir$ attached
  to  surfaces immersed in Euclidean 3--space }\label{sec:spin-trafo}

Conformal deformations of surfaces in Euclidean 3--space can be
efficiently described \cite{CPS} in terms of quaternions and Dirac
type operators. The underlying quaternionic approach to surface theory
in Euclidean 3--space was first described in \cite{KPP}, a related but
more abstract setting is developed in \cite{PP,FLPP,BFLPP}.

The \emph{quaternions} are the 4--dimensional real vector
space \[\H=\R \oplus \R i \oplus \R j \oplus \R k\] with
multiplication satisfying $ij=-ji=k$ and $i^2=j^2=-1$. The first
summand of a quaternion $x\in \H$ is called its \emph{real part}
$\Re(x)$, the $ijk$--summands together are called its \emph{imaginary
  part} $\Im(x)$. The \emph{conjugation} of a quaternion $x\in \H$
is \[\bar x= \Re(x) - \Im(x).\] The Euclidean scalar product on
$\H=\R^4$ can be written as
\[\langle x,y\rangle_{\R^4} = \Re(\bar xy).\]
In particular, the length of a quaternion is $|x|=\sqrt{\bar x x}$.
The imaginary quaternions $\Im(\H)=\R^3$ serve as our model of
Euclidean 3--space. The quaternion product
\begin{equation} \label{eq:product}
  x y= -\langle x,y\rangle_{\R^3} + x\times y
\end{equation}
of $x$, $y\in \Im(\H)=\R^3$ encodes both the scalar product $\langle
x,y\rangle_{\R^3}$ and cross product $x\times y$.  For $\lambda \in
\H_*=\H\backslash\{0\}$ we define $\Phi_\lambda(x):=\bar\lambda x
\lambda$.  Then $\Phi_\lambda$ is an orientation preserving similarity
of $\R^3=\Im(\H)$, i.e., the product of an isometry and a homothety,
and all orientation preserving linear similarities of $\R^3$ arise
this way for $\lambda\in \H_*$ unique up to sign.  In particular,
$Spin(3)\cong S^3= \{ x\in \H \mid |x|=1\}$.

Throughout the paper we denote by $M$ a compact, oriented surface with
possibly non--empty boundary.  We call two immersions $\tilde f$,
$f\colon M \rightarrow \R^3=\Im(\H)$ of $M$ \emph{topologically
  equivalent} if they belong to the same regular homotopy class.  (If
$M$ is simply connected, then any two immersions are topologically
equivalent; otherwise two immersions are topologically equivalent if
and only if they induce the same spin structure \cite{Pi}).  We call
an immersion $\tilde f\colon M \rightarrow \R^3=\Im(\H)$ a
\emph{conformal deformation} of an immersion $f\colon M \rightarrow
\R^3= \Im(\H)$ if $\tilde f$ and $f$ are topologically equivalent and
induce the same conformal structure on $M$. The following proposition
says how conformal deformations can be described in terms of Dirac
operators.

\begin{Pro}[Kamberov, Pedit, and Pinkall
  \cite{KPP}]\label{Pro:spin-trafo}
  Let $\tilde f$, $f\colon M \rightarrow \R^3=\Im(\H)$ be two immersions of
  an oriented surface $M$ that are topologically equivalent. Then
  $\tilde f$ is conformal to $f$ if and only if there exists $\lambda
  \colon M \rightarrow \H_*$ such that \[ d\tilde f = \bar \lambda df
  \lambda.\] Conversely, given an immersion $f\colon M \rightarrow
 \R^3= \Im(\H)$ and $\lambda \colon M \rightarrow \H_*$, the 1--form
  $d\tilde f = \bar \lambda df \lambda$ is closed if and only if
  \begin{equation}
    \label{eq:PreDirac}
    df\wedge d\lambda = -\rho \lambda |df|^2,
  \end{equation}
  where $|df|^2$ denotes the area form induced by $f$ and $\rho$ is a
  real valued function on $M$.
\end{Pro}
\begin{proof}
  Two immersions $\tilde f$, $f\colon M \rightarrow \R^3=\Im(\H)$
  induce the same conformal structure if and only if locally there
  exists $\lambda$ with $d\tilde f = \bar \lambda df \lambda$. The
  global existence of $\lambda$ is then equivalent to the fact that
  $f$ and $\tilde f$ induce the same spin structure (and therefore,
  according to the above terminology, are topologically equivalent and
  hence conformal deformations of each other).  The integrability
  condition takes the form \eqref{eq:PreDirac}, because \[0=d(\bar
  \lambda df \lambda) = -\Im(\bar \lambda df \wedge d\lambda)\] is
  equivalent to $\bar \lambda df \wedge d\lambda = \rho |\lambda|^2
  |df|^2$ with $\rho$ a real function.
\end{proof}

For an immersion $f\colon M \rightarrow \R^3=\Im(\H)$ of an oriented
surface $M$ we define
\begin{equation}
  \label{eq:Dirac}
  \Dir\lambda = -\frac{df\wedge d\lambda}{|df|^2}
\end{equation}
with $|df|^2$ denoting the area form induced by $f$.  The quaternionic
linear operator \[\Dir\colon C^\infty(M,\H)\rightarrow
C^\infty(M,\H)\] is formally self--adjoint and elliptic. In
Subsection~\ref{sec:std-setup} below it is shown that $\Dir$ is a
Dirac type operator. We therefore refer to $\Dir$ as the \emph{Dirac
  operator induced by the immersion}~$f$.

Using the Dirac operator $\Dir$, the integrability
equation~\eqref{eq:PreDirac} can be rewritten as
\begin{equation}\label{eq:Dirac-eq}
  \Dir \lambda = \rho \lambda.
\end{equation}
As shown in \cite{KPP}, the potential $\rho$ measures the difference
\begin{equation}
  \label{eq:mc-density-trafo}
   \tilde H |d\tilde f| = H|df| + \rho|df|
\end{equation}
between the \emph{mean curvature half--densities} $\tilde H |d\tilde
f|$ of $\tilde f$ and $H|df|$ of $f$, where for the \emph{mean
  curvature} we take the definition $H=\frac12\tr \langle
df,dN\rangle$ so that $H=1$ for the unit sphere with the orientation
such that the outside normal is positive. The mean curvature
half--density is the square root of the Willmore integrand.

As explained in the introduction, for the application of the computer
graphics algorithm proposed in \cite{CPS} one wants $\Dir-\rho$ to
have a non--empty real point spectrum.  If the underlying compact
surface $M$ has a non--empty boundary, this can be guaranteed by
posing elliptic and self--adjoint boundary conditions. In the next
section we geometrically characterize the ellipticity and
self--adjointness of \emph{local} boundary conditions for $\Dir$,
i.e., boundary conditions given by orientable subbundles $E'$ of the
trivial $\H$--bundle over $\partial M$ that prescribe the values
admissible for the restriction $\lambda_{|\partial M}$ of $\lambda$.
As it turns out, the relevant local boundary conditions are given by 
real two--dimensional bundles $E'$.

In the remainder of the section we discuss the geometry behind and
examples of the local boundary conditions for $\Dir$ given by real
two--dimensional, orientable subbundles $E'$ of the trivial
$\H$--bundle. We use that every real two--dimensional plane $E'$ in
$\H$ is of the form
\begin{equation}
  \label{eq:2d-vb}
  E' = \{\lambda\in \H \mid V \lambda = \lambda \tilde V \} 
\end{equation}
with $V$, $\tilde V\in S^2$ unique up to a common $\Z_2$--factor, see
\cite[Lemma 2]{BFLPP}.  A real two--dimensional, orientable subbundle
$E'$ of the trivial $\H$--bundle over $\partial M$ thus corresponds to
a pair of smooth maps $V$, $\tilde V\colon \partial M \rightarrow S^2$
which is unique up to a common $\Z_2$--factor (a non--orientable $E'$
would correspond to $V$, $\tilde V$ with $\Z_2$--monodromy along the
boundary). A nowhere vanishing spinor $\lambda$ satisfies the boundary
condition given by $E'$ if and only if \[\tilde V=\lambda^{-1} V
\lambda \qquad \textrm{ along } \qquad \partial M.\] This condition
can be best understood through the examples provided by the canonical
boundary conditions discussed in the following.

An immersion $f\colon M \rightarrow \R^3=\Im(\H)$ of a compact,
oriented surface $M$ with a non--empty boundary $\partial M$ is
equipped with a \emph{canonical frame} $(T,N,B)$ along $\partial M$,
where $N$ is the Gauss--map, $T$ the positive unit tangent vector
field along the boundary, and $B=T\times N$ its bi--normal field.  Let
$\tilde f$ be a conformal deformation of $f$ given by $d\tilde f =
\bar\lambda df \lambda$ with $\lambda$ a nowhere vanishing solution to
$\Dir \lambda = \rho \lambda$ for $\Dir$ the Dirac operator induced by
$f$ and $\rho$ a real valued function. Then, the canonical frame along
the boundary of $\tilde f$ is given by
\[ \tilde T =\lambda^{-1} T \lambda,\qquad \tilde N =\lambda^{-1} N
\lambda,\qquad \textrm{ and } \qquad \tilde B =\lambda^{-1} B
\lambda.\]

The canonical frame $(T,N,B)$ gives rise to three types of
\emph{canonical boundary conditions}, the boundary conditions obtained
by choosing $V=T$, $V=N$, or $V=B$, respectively. Denoting the second
vector field $\tilde V$ in \eqref{eq:2d-vb} by $\tilde V=\tilde
T$, $\tilde V=\tilde N$, or $\tilde V=\tilde B$, respectively, the
geometric meaning of the local boundary condition given by $(V,\tilde
V)$  becomes
\begin{align*}
  V=T,  \qquad & \tilde V = \tilde T  & \quad 
  \leadsto  \qquad  & \tilde f \textrm{ has prescribed } \tilde T = \lambda^{-1}
  T \lambda, or\\
  V=N,  \qquad  & \tilde V = \tilde N  & \quad \leadsto \qquad & \tilde
  f \textrm{ has prescribed } \tilde N = \lambda^{-1} N \lambda, or\\
  V=B, \qquad & \tilde V = \tilde B  & \quad  \leadsto \qquad & \tilde f
    \textrm{ has prescribed } \tilde B = \lambda^{-1} B \lambda,
\end{align*}
respectively.

\section{Elliptic and self--adjoint local boundary conditions for
  $\Dir$}\label{sec:elliptic-sa-bc}

We derive geometric conditions characterizing the ellipticity and
self--adjointness of local boundary conditions for the Dirac operator
$\Dir$ attached to immersions $f\colon M\rightarrow \R^3=\Im(\H)$ of
oriented surfaces $M$ with boundary.  Imposing elliptic and
self--adjoint boundary conditions for the Dirac operator $\Dir$
assures that $\Dir-\rho$ with $\rho$ a real function behaves
essentially as in the case of compact surfaces without boundary (cf.\
Subsection~\ref{sec:consequences}).

\subsection{Conditions for ellipticity and
  self--adjointness of local boundary conditions}\label{sec:main-thms}
By a \emph{local boundary condition} for the Dirac operator $\Dir$ in
\eqref{eq:Dirac} we mean a smooth, orientable real subbundle $E'$ of
the trivial $\H$--bundle over $\partial M$ that prescribes the
admissible values of the restriction $\lambda_{|\partial M}$ of
$\lambda$ to the boundary $\partial M$.  (The assumption that $E'$ is
orientable is included to simplify the notation;
Theorems~\ref{Thm:elliptic} and \ref{Thm:selfadjoint} below hold also
in the non--orientable case, although the vector fields $V$ and
$\tilde V$ then have $\Z_2$--monodromy.)

The following two theorems imply that elliptic or self--adjoint
local boundary conditions are necessarily given by two--dimensional subbundles
$E'$. As explained in Section~\ref{sec:spin-trafo}, every real
two--dimensional, orientable subbundle $E'$ of the trivial
$\H$--bundle over $\partial M$ is given by a pair of vector fields
$V$, $\tilde V\colon \partial M \rightarrow S^2\subset \Im(\H)$ along
$\partial M$ for which
\begin{equation}
  \label{eq:boundary-vfd}
   E'_p = \{\lambda \in \H \mid V(p) \lambda = \lambda \tilde
V(p)\}.
\end{equation}

\begin{Thm}\label{Thm:elliptic} Let  $f\colon M
  \rightarrow \R^3=\Im(\H)$ be a conformal immersion of a compact
  Riemann surface $M$ with boundary. A local boundary condition $E'$
  for the Dirac operator \eqref{eq:Dirac} induced by $f$ is
  \textbf{elliptic} if and only if $E'$ is a two--dimensional real
  orientable vector bundle such that the corresponding maps $V$,
  $\tilde V\colon \partial M\rightarrow S^2\subset \Im(\H)$ satisfy
  \[ V(p) \neq \pm N(p) \qquad \textrm{for all} \quad p \in \partial M, \]
  where $N$ is the Gauss map $N\colon M \rightarrow S^2$ of $f$.
\end{Thm}

\begin{Thm}\label{Thm:selfadjoint}
  Let $f\colon M \rightarrow \R^3=\Im(\H)$ be a conformal immersion of
  a compact Riemann surface $M$ with boundary. A local boundary
  condition $E'$ for the Dirac operator \eqref{eq:Dirac} induced by
  $f$ is \textbf{self--adjoint} if and only if $E'$ is a
  two--dimensional real orientable vector bundle such that the
  corresponding maps $V$, $\tilde V\colon \partial M\rightarrow
  S^2\subset \Im(\H)$ satisfy
  \[ V(p) \perp  T(p) \qquad \textrm{for all} \quad p \in \partial
  M, \] where $T$ denotes the positive unit tangent field of $f$ along
  the boundary $\partial M$ of $M$.
\end{Thm}

As an immediate consequence of Theorems~\ref{Thm:elliptic} and
\ref{Thm:selfadjoint} we obtain for the three canonical boundary
conditions (see the end of Section~\ref{sec:spin-trafo}) that
\begin{itemize}
\item the bi--normal boundary condition which prescribes $\tilde B$
 is elliptic and self--adjoint,
\item the tangential boundary condition which prescribes $\tilde T$ is
  elliptic, but not self--adjoint, and
\item the normal boundary condition which prescribes $\tilde N$ is
  self--adjoint, but not elliptic.
\end{itemize}

The proofs of Theorems~\ref{Thm:elliptic} and~\ref{Thm:selfadjoint}
are given in \ref{sec:thm1-elliptic} and \ref{sec:thm2-sa} below.  In
order to make them more accessible to Readers not familiar with
elliptic boundary value problems, we give detailed references to the
excellent, self--contained text~\cite{BB} by C.~B\"ar and W.~Ballmann
on first order elliptic boundary value problems.  In particular, in
\ref{sec:general-bc} we review the necessary notation of \cite{BB} and
in \ref{sec:std-setup} we explain that we are in the ``standard
setup'' of \cite{BB}.  We follow the notation of \cite{BB} with two
notable differences:
\begin{itemize}
\item In the following analysis all linear operators, even if they are
  quaternionic linear, are viewed as \textbf{real} instead of complex
  linear operators; in other words, the complex structure obtained by
  restriction to the complex scalar field is \textbf{not} used. This
  might seem confusing at first thought, but is geometrically
  necessary if one wants to allow for real subbundles $E'$ as boundary
  conditions.  (At the only point where we actually need the complex
  scalar field, we simply complexify the whole setup, cf.\ the proof of
  Theorem~\ref{Thm:elliptic} in Section~\ref{sec:thm1-elliptic}.)
\item In contrast to \cite{BB} we are exclusively interested in the
  case that $M$ is \textbf{compact} with boundary and apply
  corresponding simplifications of the notation. In particular, we
  denote by $C^\infty(M,E)$ the space of smooth section of $M$ and by
  $C^\infty_0(M,E)$ the space of smooth sections with support in
  $M\backslash \partial M$ (for which \cite{BB} would instead write
  $C^\infty_c(M,E)$ and $C^\infty_{cc}(M,E)$, respectively).
\end{itemize}

\subsection{General theory of local boundary conditions for first
  order elliptic operators (following
  \cite{BB})}\label{sec:general-bc} Let
$\Dir\colon C^\infty(M,E)\rightarrow C^\infty(M,F)$ be a first order
elliptic operator as in the standard setup 1.5 of~\cite{BB}.  We view
$\Dir$ as a densely defined unbounded operator $\Dir_0$ with domain
$\dom(\Dir_0)=C^\infty_0(M,E)$ in $L^2(M,E)$. In the case that
$\partial M=\emptyset$ there is one distinguished extension of
$\Dir_0$, its closure defined on the first Sobolev space.  In the case
$\partial M \neq \emptyset$ which we are interested in, there are many
extensions of $\Dir_0$ corresponding to the different boundary
conditions.

The so called \emph{maximal extension} $\Dir_{max}$ of $\Dir_0$ (cf.\
p.4 of \cite{BB}) has as its domain $\dom(\Dir_{max})$ the space of
all $\phi\in L^2(M,E)$ for which $\Dir\phi$ exists in the distribution
sense and $\Dir\phi \in L^2(M,E)$, i.e., $\phi\in \dom(\Dir_{max})$ if
and only if $\phi \in L^2(M,E)$ and there is $\xi\in L^2(M,F)$ such
that \[\langle\phi,\Dir^*\psi\rangle= \langle\xi,\psi\rangle\] for all
$\psi\in C^\infty_0(M,F)$.  Here $\Dir^* \colon
C^\infty(M,F)\rightarrow C^\infty(M,E)$ denotes the formal adjoint of
$\Dir$, the operator characterized by
\[ \langle \Dir \phi,\psi \rangle = \langle \phi,\Dir^* \psi\rangle \]
for all $\phi\in C^\infty_0(M,E)$ and $\psi\in C^\infty_0(M,F)$.  More
elegantly, one can characterize $\Dir_{max}$ as the adjoint
$(\Dir^*)^{ad}$ of the formal adjoint $\Dir^*$ in the unbounded
operator sense. This immediately shows that $\Dir_{max}$ is a closed
extension of $\Dir_0$. In particular, the graph norm makes
$\dom(\Dir_{max})$ into a Hilbert space.

There are two equivalent ways of describing boundary conditions for
$\Dir$:
\begin{enumerate}
\item by prescribing boundary values of sections of $E$,
\item by prescribing a closed extension of $\Dir_0$ contained in
  $\Dir_{max}$.
\end{enumerate}
The first point (1) is technically more involved: by Theorem~1.7 or
6.7 in \cite{BB}, the space $C^\infty(M,E)$ is dense in
$\dom(\Dir_{max})$ with respect to the graph norm and the restriction
map $\mathcal{R}\colon C^\infty(M,E)\rightarrow C^\infty(\partial
M,E)$, $\phi\mapsto \phi_{|\partial M}$ has a unique continuous
extension to $\dom(\Dir_{max})$ which is a surjective map to the space
$\check H(A)$ (defined in (3) or (36) of \cite{BB}). A \emph{boundary
  condition} can then be defined as a closed subspace $B$ of $\check
H(A)$ (see Def.~1.9 or 7.1 of \cite{BB}).  Such a boundary condition
$B$ defines a closed extension $\Dir_{B,max}$ of $\Dir_0$ contained in
$\Dir_{max}$, i.e.,
\[ \Dir_0\subset \Dir_{B,max}\subset \Dir_{max} \] with
domain \[\dom(\Dir_{B,max})= \{ \phi\in \dom(\Dir_{max})\mid
\mathcal{R}\phi \in B\},\] see (4) or 7.1 of \cite{BB}. Conversely, by
Prop.~7.2 of \cite{BB}, every closed extension of $\Dir_0$ contained
in $\Dir_{max}$ is of the form $\Dir_{B,max}$ for some boundary
condition $B$, showing that boundary conditions (1) can be
equivalently described by closed extensions (2).  For example, the
closure of $\Dir_0$, the so called \emph{minimal extension}
$\Dir_{min}$, corresponds to the Dirichlet boundary condition
$B=\{0\}$.

A \emph{local boundary condition} corresponding to a subbundle $E'$ of
$E$ is obtained by taking for $B$ the closure of $C^{\infty}(\partial
M,E')$ in $\check H(A)$ (this slightly differs from Definition~7.19 in
\cite{BB}, but is equivalent in the elliptic case, cf.\ Lemma~7.10 and
the following remark in \cite{BB}).  For example, $E'=\{0\}$ and
$E'=E$ yield the minimal and maximal extensions.

The adjoint of an operator $\Dir_{B,max}$ corresponding to a boundary
condition $B$ for $\Dir$ is the operator
\[(\Dir_{B,max})^{ad}=\Dir^*_{B^{ad},max}\] for $B^{ad}$ the
\emph{adjoint boundary condition} (see Section 7.2 of \cite{BB}) given by
\begin{equation}
  \label{eq:adjoint-bc}
  B^{ad} = \{ \psi\in \check H(\tilde A)\mid  \langle \sigma_0
  \varphi,\psi\rangle=0 \textrm{ for all } \varphi \in B \},
\end{equation}
 see (6) or (63)
in \cite{BB}, because, by (48) there, for all $\varphi \in
\dom(\Dir_{max})$ and $\psi \in \dom(\Dir^*_{max})$ 
\[  \langle\Dir_{max}\varphi,\psi\rangle_{L^2(M)} -  \langle\varphi,(\Dir^*)_{max}\psi\rangle_{L^2(M)} = - \langle
\sigma_0\mathcal{R} \varphi,\mathcal{R}\psi\rangle_{L^2(\partial M)} .\]

The adjoint boundary condition of the local boundary condition given
by $E'$ is again local and given by $F'=(\sigma_0(E'))^\perp$ (for
example, the minimal and maximal boundary conditions are adjoint to
each other). To see this we have to check that $B^{ad}$ is the closure
of $C^{\infty}(\partial M,(\sigma_0(E'))^\perp)$. It is clear by
\eqref{eq:adjoint-bc} that the closure of $C^{\infty}(\partial
M,(\sigma_0(E'))^\perp)$ is contained in $B^{ad}$. They coincide by
Lemma~6.3 in \cite{BB}, because under the perfect pairing $\beta$ the
decomposition of the dense subspace $C^\infty(\partial
M,E)=C^\infty(\partial M,E')\oplus C^\infty(\partial M,(E')^\perp)$ of
$\check H(A)$ induces a direct sum decomposition of the space $\check
H(\tilde A)$ into closed subspaces one of which is the closure of
$C^{\infty}(\partial M,(\sigma_0(E'))^\perp)$.

This immediately yields a condition for the operator $\Dir_{B,max}$
corresponding to a local boundary condition to be self--adjoint. For
this, the underlying differential operator $\Dir$ has to be
\emph{formally self--adjoint} which means that $\Dir=\Dir^*$ (and
hence necessarily $E=F$), i.e.,
\[ \langle\Dir \phi,\psi \rangle = \langle\phi,\Dir\psi\rangle \] for
all $\phi$, $\psi\in C^\infty_0(M,E)$. In the formally self--adjoint
case, the adjoint boundary condition $B^{ad}$ of a given boundary
condition $B$ is a boundary condition for the same differential
operator~$\Dir$. The extension $\Dir_{B,max}$ is then
\emph{self--adjoint} if and only if $B=B^{ad}$. In particular, a local
boundary condition given by $E'$ is self--adjoint if and only if
\begin{equation}
  \label{eq:self-adj-bc}
  E'=(\sigma_0(E'))^\perp.
\end{equation}

Even more important than self--adjointness for what follows is the
ellipticity of boundary conditions which implies the Fredholm property
of the extension $\Dir_{B,max}$ (see Subsection~\ref{sec:consequences}
for references): a boundary condition is \emph{elliptic} if and only
if
\[ \dom(\Dir_{B,max}) \subset H^1_{loc}(M,E) \quad \textrm{ and }
\quad \dom(\Dir^*_{B^{ad},max}) \subset H^1_{loc}(M,F) \] (see
Definition~1.10 in \cite{BB} and also Remark~1.11 and Definition~7.5
for alternative Definitions which are equivalent by Theorems~1.12 and
7.11).

\subsection{Adaption to ``standard setup'' described in 1.5 of \cite{BB}:}
\label{sec:std-setup} 
In the rest of the paper $\Dir$ is the Dirac operator \eqref{eq:Dirac}
induced by an immersion $f\colon M \rightarrow \R^3=\Im(\H)$ of an
oriented surface $M$ with boundary. The immersion $f$ also defines a
volume form on $M$ which enters into the definition of $L^2$-- and
Sobolev norms (although on a compact manifold changing the volume form
yields equivalent $L^2$-- and Sobolev norms).

To make contact with the ``standard setup'' \cite[1.5]{BB}, we denote
by $E=F$ the trivial $\R^4=\H$--bundle over $M$ with scalar product
$\langle\lambda,\mu \rangle=\Re(\bar \lambda \mu)$ and view our Dirac
operator $\Dir$ as a real linear map $\Dir\colon C^\infty(M,E)
\rightarrow C^\infty(M,F)$ between smooth sections of $E=F$.  The
operator $\Dir$ is elliptic, because its symbol \cite[(21)]{BB} is
\begin{equation}
  \label{eq:symbol}
  \sigma_\Dir(\xi)(\lambda) = -\frac{df \wedge \xi}{|df|^2} \lambda 
\end{equation}
for $\xi\in T^*_pM$ and $\lambda \in E_p=\H$ so that
$\sigma_\Dir(\xi)$ is an isomorphism for every $\xi\neq 0$.

To see that we are in the standard setup, we check that $\Dir$ is
a Dirac type operator in the sense of \cite[Example 4.3(a)]{BB}.
Because
\[\sigma_\Dir(\xi)^* = - \sigma_{\Dir^*}(\xi)\]
(see \cite[(22)]{BB}) and $\Dir$ is formally self--adjoint,
$\Dir=\Dir^*$, the condition that $\Dir$ is a Dirac type operator reads
\[ \sigma_\Dir(\xi) \sigma_\Dir(\eta) + \sigma_\Dir(\eta)
\sigma_\Dir(\xi) =-2 \langle\xi,\eta\rangle \Id_E \] for all $\xi$,
$\eta\in T^*_pM$ and $p\in M$.  This condition can be easily verified,
e.g.\ by taking $\xi$, $\eta\in \{ dx, dy\}$ for $z=x+iy$ some local
holomorphic chart.  By Example 1.6 of \cite{BB}, the fact that $\Dir$
is a Dirac type operator implies that we are in the standard setup.

We compute now the normal form
\[\Dir=\sigma_t\left(\frac{\partial}{\partial t} + D_t\right)\]
(cf.\ (2) and the proof of Lemma~4.1 in \cite{BB}) with respect to the
choice of an inner normal field perpendicular to the boundary (in
\cite{BB} this would be called $T$) and compatible coordinates defined
on a collar of the boundary. On each component of the boundary we fix
a holomorphic chart mapping a collar of the boundary to a strip
$r<|z|\leq 1$ (to see that this can be done, glue a neighborhood of
the boundary into the sphere, apply uniformization, and use the
Riemann mapping theorem for simply connected domains with smooth
boundary). Writing $z=\exp(i\theta -t)$, we obtain a parametrization
of the boundary by $\theta$ with $t=0$ and $\frac\partial{\partial t}$
is a normal field perpendicular to the boundary whose length coincides
with the length of $\frac\partial{\partial \theta}$. In these
coordinates, our operator $\Dir$ takes the form
\begin{equation}
  \label{eq:normal-form}
  \Dir = -\frac1{|f_\theta|^2}  \left(f_\theta \frac\partial{\partial t} - f_t \frac\partial{\partial \theta}\right) = -\frac {f_\theta}{|f_\theta|^2}  \left( \frac\partial{\partial t} - f_\theta^{-1}f_t \frac\partial{\partial \theta}\right).
\end{equation}
The endomorphism field $\sigma_0$ in Definition~1.4 of \cite{BB} is
then 
\begin{equation}
  \label{eq:sigma0}
  \sigma_0(\lambda)= -\frac {f_\theta}{|f_\theta|^2} \lambda
\end{equation}
and the adapted first order operator $A$ on the boundary in
\cite[(2)]{BB} is
\begin{equation}
  \label{eq:adapted-operator-A}
  A=- f_\theta^{-1}f_t \frac\partial{\partial \theta}=
  N\frac\partial{\partial \theta}.
\end{equation}

\subsection{Proof of Theorem~\ref{Thm:elliptic} on ellipticity of
  local boundary conditions}\label{sec:thm1-elliptic}

We apply the ellipticity criterion for local boundary conditions given
in Theorem~7.20~(iv) of \cite{BB}. For this we have to complexify our
setting and pass to the complexified bundles $E^\C$, $(E')^\C$... and
operators $\Dir^\C$ and $A^\C$. The criterion then says that
ellipticity of a local boundary condition given by a subbundle
$E'\subset E$ is equivalent to the property that the orthogonal
projection to $(E')^\C$ pointwise restricts to an isomorphism between
the bundle $U$ spanned by the negative eigenspaces of $J
\sigma_A(\xi)$ and $(E')^\C$, where $J$ denotes the complex structure
of $E^\C$.  This is equivalent to the fact that the bundles $U$ and
$(E')^\C$ have the same dimensions and $U \cap ((E')^\C)^\perp=\{0\}$.

The bundle $E^\C$ is the trivial bundle $\H^2$ seen as a real bundle
with complex structure $J(\lambda,\mu) = (-\mu,\lambda)$. From
\eqref{eq:adapted-operator-A} we obtain that
$\sigma_A(\xi)(\lambda)=N\lambda$ for $\xi=d\theta$ so that the only
eigenvalues of $J\sigma_A(\xi)$ are $\pm 1$, each with an eigenspace
of complex dimension two. Its negative eigenvectors $(\lambda,\mu)\in
U_p$ are therefore characterized by the equations
\[ N(p)\lambda = -\mu \qquad \textrm{ and } \qquad N(p)\mu=\lambda \]
which are equivalent, because $N^2=-1$. For elliptic boundary
conditions, the bundle $(E')^\C$ is thus (complex) 2--dimensional so
that the underlying real bundle $E'$ is of the form $E'_p=\{
\lambda\mid V(p)\lambda = \lambda \tilde V(p)\}$ for $V$, $\tilde V
\colon \partial M\rightarrow S^2$, cf.\ \eqref{eq:2d-vb}. The
perpendicular space $((E')^\C)^\perp$ with respect to the metric
$\langle(\lambda_1,\mu_1),(\lambda_2,\mu_2)\rangle=
\Re(\bar\lambda_1\lambda_2) + \Re(\bar\mu_1\mu_2)$ is then given by
\[((E')^\C)^\perp = \{ (\lambda,\mu)\mid V(p)\lambda =- \lambda \tilde
V(p)\textrm{ and } V(p)\mu =- \mu \tilde V(p)\}.\] Assume now that
$(\lambda,\mu)\in U_p \cap ((E')^\C_p)^\perp$ is non--trivial.
Plugging $N(p)\lambda=-\mu$ into $V(p)\mu =- \mu \tilde V(p)$ yields
\[V(p)N(p)\lambda =- N(p)\lambda \tilde V(p).\] On the other hand,
multiplying $V(p)\lambda =- \lambda \tilde V(p)$ from the left by
$N(p)$ we obtain \[N(p) V(p)\lambda =- N(p)\lambda \tilde V(p).\]
Since $\lambda\neq 0$, comparing the left hand sides of both equations
yields $V(p)N(p)=N(p)V(p)$ which is equivalent to $V(p)= \pm N(p)$,
because both $V$ and $N$ take values in $S^2\subset \Im(\H)$ and, by
\eqref{eq:product}, two imaginary quaternions commute if and only if
they are real linearly dependent. Conversely, if $V(p)= \pm N(p)$ one
can find non--trivial $(\lambda,\mu)\in U_p \cap ((E')^\C_p)^\perp$,
since one can simultaneously solve the preceding two equations. This
completes the proof.

\subsection{Proof of Theorem~\ref{Thm:selfadjoint} on
  self--adjointness of local boundary conditions}\label{sec:thm2-sa}

Since $\Dir$ is formally self--adjoint, a local boundary condition
given by a real subbundle $E'$ of the trivial $\H$--bundle over
$\partial M$ is self--adjoint if and only if
$E'=(\sigma_0(E'))^\perp$, see \eqref{eq:self-adj-bc}. This
immediately shows that the boundary condition can only be
self--adjoint if $E'$ is two--dimensional. Assuming that $E'$ is
two--dimensional, there are $V$, $\tilde V\colon \partial M
\rightarrow S^2$ with \[E'_p=\{ \lambda \in \H\mid V(p)\lambda =
\lambda \tilde V(p)\}.\] Then, by \eqref{eq:sigma0},
\[\sigma_0 (E'_p)=\{
\lambda \in \H\mid T(p)V(p)T(p)^{-1}\lambda = \lambda \tilde V(p)\},\]
because up to some negative real factor $\sigma_0$ acts by
left--multiplication with $T$. Hence \[(\sigma_0 (E'_p))^\perp=\{
\lambda \in \H\mid -T(p)V(p)T(p)^{-1}\lambda = \lambda \tilde V(p)\}\]
so that the boundary condition is self--adjoint if and only if $V=-T V
T^{-1}$. But this is equivalent to $V$ being pointwise perpendicular
to $T$, see \eqref{eq:product}.

\subsection{Consequences of ellipticity and self--adjointness of
  boundary conditions}\label{sec:consequences}

By the general theory of first order elliptic boundary value problems,
on a surfaces with boundary the Dirac operator $\Dir-\rho$ with real
potential $\rho$ has a similar behavior as in the case of empty
boundary if one imposes elliptic and self--adjoint boundary conditions
for $\Dir$ (the order zero perturbation introduced by subtracting the
real function $\rho$ changes neither the ellipticity nor the
self--adjointness of the boundary condition):

\begin{itemize}
\item (Fredholm property) The extension $(\Dir-\rho)_{B,max}$ of
  $\Dir-\rho$ corresponding to an elliptic boundary condition $B$ is a
  Fredholm operator (see Theorem~1.18 resp.\ Theorem~8.5 of \cite{BB}
  and Remark~8.1 there for why both versions of the theorem are
  equivalent; the completeness and coercivity assumptions are
  automatically satisfied, since in our case $M$ is assumed to be
  compact, cf.\ Definition 1.1 and Example~8.3 of~\cite{BB}).
\item (Regularity) For local elliptic boundary conditions, the kernel
  of $(\Dir-\rho)_{B,max}$ is contained in the space $C^\infty(M,\H)$ of
  functions that are smooth up to the boundary (see Proposition 7.24
  and Corollary 7.18 of \cite{BB}).
\item (Spectrum) Because the ellipticity of a boundary problem is not
  affected by lower order deformations, for every $\mu \in \C$ the
  operator $(\Dir-\rho)_{B,max}-\mu\Id$ is Fredholm (that is,
  $\Dir-\rho$ has no essential spectrum) and its kernel, the space of
  eigenspinors of $\Dir-\rho$ to eigenvalue $\mu$, consists of smooth
  functions. In particular, if $(\Dir-\rho)_{B,max}$ has index zero,
  its spectrum is a pure point spectrum which is either a discrete set
  or all of $\C$ (the latter can be seen by looking at the determinant
  for the holomorphic family $(\Dir-\rho)_{B,max}-\mu\Id$, $\mu\in
  \C$, of Fredholm operators).
\item (Self--adjointness) If in addition to ellipticity the boundary
  condition is self--adjoint, the spectral theorem for self--adjoint
  operators implies reality of the spectrum and the existence of an
  orthonormal basis of smooth eigenspinors.
\end{itemize}

\section{Fredholm index of $\Dir$ with elliptic boundary condition
  (for $M$ a disc)}\label{sec:index}

The prescription of an elliptic local boundary condition extends the
Dirac operator $\Dir$ attached to a conformal immersion $f$ to an
operator of Fredholm type. We compute the index of such a Fredholm
operator in the case that the underlying surface $M$ is a disc.
Homotopy invariance allows to reduce this computation to a classical
result by I.N.~Vekua.

An elliptic local boundary condition for the Dirac operator $\Dir$
attached to an immersed disc $f\colon M\rightarrow \R^3=\Im(\H)$ is
given by functions $V$, $\tilde V\colon S^1=\partial M \rightarrow
S^2$ such that $\pm V$ nowhere coincides with the Gauss map $N$ of
$f$. Viewing $V$ as a section of the sphere--bundle $S^2\backslash
\{\pm N\}$ with north-- and south--pole sections removed allows to
define the \emph{degree} of $V$, because $S^2\backslash \{\pm N\}$ is
homotopy equivalent to a trivial $S^1$--bundle. More precisely, with
respect to the canonical frame $T,N,B$ along the boundary of the
immersion, the vector field $V$ can be written as
\[ V= v_1 T + v_2 N + v_3 B \] for $v_i\in
C^\infty(\partial M,\R)$. Because $v_1^2+v_2^2+v_3^2=1$ and $v_2(p)\neq
\pm 1$ for all $p\in \partial M$, the degree of $V$ can be defined 
\begin{equation}
  \label{eq:degreeV}
   \deg(V) := \deg(v_1,v_3)
\end{equation}
as the mapping degree of the plane curve $(v_1,v_3)\colon
S^1\cong \partial M \rightarrow \R^2\backslash \{0\}$.

\begin{Thm}\label{Thm:index}
  The index of the Dirac operator $\Dir$ for an immersed disc $f\colon
  M\rightarrow \R^3=\Im(\H)$ with elliptic local boundary condition
  given by functions $V$, $\tilde V\colon S^1=\partial M \rightarrow
  S^2$ is
  \[ \Index\left(\Dir_{(V,\tilde V)}\right)= 2\deg(V)\] with $\deg(V)$ defined by
  \eqref{eq:degreeV}.
\end{Thm}

The theorem is essentially a quaternionified version of the following
result by I.N.~Vekua (see p.~118 in \cite{BoBl} or p.~266 in
\cite{H}): the index of the operator \[(\Delta,\mathcal{R})\colon
C^\infty(D,\R)\rightarrow C^{\infty}(D,\R)\oplus C^{\infty}(S^1,\R)\]
extended to suitable Sobolev spaces is
\begin{equation}
  \label{eq:index_laplace}
  \Index(\Delta,\mathcal{R})=2-2p
\end{equation}
when $\Delta=\frac{\partial^2}{\partial
  x^2}+\frac{\partial^2}{\partial y^2}$ is the Laplace operator on
$D=\{z=x+iy\in \C=\R^2\mid |z|\leq 1\}$ and $\mathcal{R}(f)= \Re(\nu
\cdot \partial f)$ with $\partial f= \frac12 \left(\frac{\partial
    f}{\partial x} - i \frac{\partial f}{\partial y}\right)$ and
$\nu(z)=z^p$ for $p\in \Z$. More precisely,
$\ker(\Delta,\mathcal{R})$ is 1--dimensional if $p>0$ and
$2-2p$--dimensional if $p\leq 0$, while $\coker(\Delta,\mathcal{R})$ is
$2p-1$--dimensional if $p>0$ and $0$--dimensional if $p\leq 0$.

We show now that this is equivalent to the fact that the index of the
(real linear) operator \[(\bar\partial ,\mathcal{R'})\colon
C^\infty(D,\C)\rightarrow C^{\infty}(D,\C)\oplus C^{\infty}(S^1,\R)\]
extended to the right Sobolev spaces is
\begin{equation}
  \label{eq:index_dbar}
  \Index(\bar\partial,\mathcal{R'})=1-2p,
\end{equation}
where $\bar \partial \lambda= \frac12 \left(\frac{\partial
    \lambda}{\partial x} + i \frac{\partial \lambda}{\partial
    y}\right)$, $\mathcal{R'}(\lambda)= \Re(\nu \cdot \lambda)$, and
as above $\nu(z)=z^p$ for $p\in \Z$.  More precisely,
$\ker(\bar\partial,\mathcal{R'})$ is 0--dimensional if $p>0$ and
$1-2p$--dimensional if $p\leq 0$, while
$\coker(\bar\partial,\mathcal{R'})$ is $2p-1$--dimensional if $p>0$
and $0$--dimensional if $p\leq 0$.  To see that \eqref{eq:index_dbar}
is equivalent to \eqref{eq:index_laplace}, we use the following lemma
(cf.\ Appendix~A in \cite{BB}):
\begin{Lem}\label{lem:index}
  Let $A=(B,C)\colon H\rightarrow H_1\oplus H_2$ be a bounded linear
  operator between Hilbert spaces such that $B$ is surjective. Then
  $A$ is Fredholm if and only if $C_{|\ker(B)}$ is Fredholm
  and \[\Index(A) = \Index(C_{|\ker(B)}).\] More precisely, the
  kernels and cokernels of $A$ and $C_{|\ker(B)}$ have the same
  dimensions.
\end{Lem}
\begin{proof}
  With respect to the decomposition $H=\ker(B)^\perp \oplus \ker(B)$,
  the operator $A$ takes the form
  \[ A=
  \begin{pmatrix}
    B_{|\ker(B)^\perp} & 0 \\
    C_{|\ker(B)^\perp} & C_{|\ker(B)} 
  \end{pmatrix}=: \begin{pmatrix}
    B' & 0 \\
    C' & C'' 
  \end{pmatrix},\] where $B'=B_{|\ker(B)^\perp}$ is continuously
  invertible by the open mapping theorem. Now
\[  
  \begin{pmatrix}
    Id_{H_1} & 0 \\
    -C' (B')^{-1} & Id_{|\ker(B)}
  \end{pmatrix} \begin{pmatrix}
    B' & 0 \\
    C' & C''
  \end{pmatrix} \begin{pmatrix}
    (B')^{-1} & 0 \\
    0 & Id_{|\ker(B)}
  \end{pmatrix} = \begin{pmatrix}
    Id_{H_1} & 0 \\
    0 & C''
  \end{pmatrix} \] from which one can readily read of the claim.
\end{proof}
Because both $\Delta$ and $\bar\partial$ are surjective (e.g.\ by
unique solvability of the Dirichlet problem for
$\Delta=4\bar\partial\partial$), Lemma~\ref{lem:index} implies that
\[ \Index(\Delta,\mathcal{R}) = \Index(\mathcal{R}_{|\ker(\Delta)}) \]
and
\[ \Index(\bar\partial,\mathcal{R}') =
\Index(\mathcal{R}_{|\ker(\bar\partial)}'). \] Equivalence of
\eqref{eq:index_dbar} and \eqref{eq:index_laplace} now follows from
$\Index(\mathcal{R}_{|\ker(\bar\partial)}')=\Index(\mathcal{R}_{|\ker(\Delta)})-1$,
because every holomorphic function $\lambda$ on $D$ is of the form
$\lambda = \partial f$ for some real valued harmonic function $f$
which is unique up to adding a real constant. (To see this, note that
every holomorphic function $\lambda$ on $D$ has a holomorphic
primitive $\tilde\lambda$ with $\partial \tilde \lambda =
\lambda$. Writing $\tilde \lambda = \frac12 (f+ig)$ we obtain $\lambda
= \partial f$, because $\partial f = i\partial g$.)

\begin{proof}[Proof of Theorem~\ref{Thm:index}]
  Because $\mathcal{R}'\colon C^\infty(D,\C)\rightarrow
  C^{\infty}(S^1,\R)$, $\lambda \mapsto \Re(\nu\cdot \lambda)$ with
  $\nu(z)=z^p$, $p\in \Z$ is surjective, Lemma~\ref{lem:index} implies
  that \eqref{eq:index_dbar} is equivalent to
  \[ \Index(\dbar_{|\ker(\mathcal{R}')})= 1-2p. \] In particular,
  taking the direct sum of $\dbar$ with itself and boundary conditions
  $\nu_i(z)=z^{p_i}$ with $p_i\in\Z$, $i=1$, $2$ yields and operator with index
  \[ \Index\left((\dbar \oplus \dbar)_{|\ker(\mathcal{R}'_1)\oplus
    \ker(\mathcal{R}'_2)}\right) = 2-2(p_1+p_2). \]

Now the operator $\dbar\oplus \dbar$ is essentially the Dirac operator
$\Dir$ belonging to the immersion $f(z)=j z$ of the disc $D$ into
$\R^3=\Im(\H)$, because for $\lambda_1$, $\lambda_2\in C^\infty(D,\C)$
we have \[\Dir (\lambda_1+\lambda_2 j)= -\frac{df\wedge d\bar
  z}{|dz|^2} (\dbar \lambda_1+ \dbar\lambda_2 j).\] The elliptic
boundary condition $\lambda_i\in \Gamma(\ker(\mathcal{R}'_i))$, $i=1$, $2$ on
$\lambda=\lambda_1+\lambda_2 j$ given by $\nu_1$ and $\nu_2$
corresponds to the two--dimensional subbundle $E'$ of $\H$ with
sections $i\bar \nu_1$ and $i\bar \nu_2 j$. The vector fields defining
$E'$ are thus
  \begin{equation}
    \label{eq:quaternionic-vekua}
    V(z) = -j \nu_1\nu_2 \qquad \textrm{ and } \qquad \tilde V(z) =
    j   \bar \nu_1 \nu_2.
  \end{equation}
  The canonical frame along the boundary of $f$ is given by $T(z)=ji\,
  z$, $N(z) =-i$, and $B(z)=j\, z$ so that
  \[ V(z) = T(z) i \nu_1\nu_2 z^{-1} \] and $V$ has degree $\deg(V) =
  1-p_1-p_2$ (note that $B=-Ti$ so that by \eqref{eq:degreeV} the
  degree of $V$ equals the degree of $\overline{i \nu_1\nu_2
    z^{-1}}$).
  
  Thus, for the operator $\Dir$ belonging to $f(z)=jz$ with
  boundary condition given by $V$ and $\tilde V$ above we indeed have
  \[ \Index\left(\Dir_{(V,\tilde V)}\right)= 2\deg(V).\] This proves
  the theorem, because by homotopy invariance of the index it is
  sufficient to verify it for one immersed disc and one boundary
  condition of each degree.
\end{proof}

\section{Spectral flow of $\Dir$ with periodic families of
  self--adjoint, elliptic boundary conditions (for $M$ a
  disc)}\label{sec:spectral-flow}

For self--adjoint, elliptic boundary conditions, the Fredholm index is
always zero and hence not a very interesting invariant. Instead, there
is another invariant for families of self--adjoint, elliptic boundary
conditions (or, more generally, for families of self--adjoint Fredholm
operators), the spectral flow first considered by M.F.~Atiyah and
G.~Lusztig (cf.\ p.~93 of \cite{APS}). Because the treatment in
\cite{APS} is very brief, especially when it comes to the case of
unbounded operators, the Reader might wish to consult \cite{P,BLP} for
an alternative approach including a detailed discussion (\cite{BLP})
of the unbounded case which takes into account the possibility of
varying domains.
 
The idea behind the concept of \emph{spectral flow} is the following:
given a 1--parameter family of self--adjoint Fredholm operators $F_t$,
$t\in[a,b]$, its spectral flow is defined as the number of eigenvalues
counted with multiplicities that flow from $\R_{<0}$ to $\R_{\geq 0}$
when $t$ goes from~$a$ to~$b$ (see \cite{BLP} for details).  Well
definedness of the spectral flow can be most easily understood if one
assumes that $F_t-\mu \Id$ is Fredholm for all $\mu\in \R$ (which is
always the case for elliptic operators with self--adjoint, elliptic
boundary condition), because the spectra of the $F_t$ are then
discrete series of real numbers varying continuously with $t$ (cf.\
\cite[IV,3.5]{K}).

An important property (cf.\ \cite{APS,P} for the bounded and
\cite{BLP} for the unbounded case) of spectral flow is that it is
invariant under homotopy with fixed endpoints and hence defines a
homomorphism
\begin{equation}
  \label{eq:spectralflow}
  sf\colon \pi_{\leq 1}(C\mathcal{F}^{SA}) \rightarrow \Z 
\end{equation}
from the fundamental groupoid of the self--adjoint part
$C\mathcal{F}^{SA}$ in the space of closed Fredholm operators
$C\mathcal{F}$ to the entire numbers $\Z$.  It appears to be unknown
whether \eqref{eq:spectralflow} is injective and hence an isomorphism
(see the introduction to \cite{BLP}).  This is in contrast to the case
of bounded, self--adjoint Fredholm operators, because the restriction
of \eqref{eq:spectralflow} to the non--trivial component of the space
of bounded, self--adjoint Fredholm operators is known \cite{APS,P} to
be an isomorphism (see \cite{P} and the references therein for
applications to K--theory).

If the underlying surface $M$ is a disc, the space of immersions
$f\colon M \rightarrow \R^3=\Im(\H)$ is connected and simply
connected.  However, the space of self--adjoint Fredholm operators
obtained from self--adjoint, elliptic boundary conditions for Dirac
operators $\Dir$ induced by immersions $f$ of the disc $M$ is
connected, but not simply connected. The spectral flow for loops of
such operators is computed in the following theorem.

\begin{Thm}\label{Thm:spectralflow}
  Let $\Dir_t$, $t\in S^1$ be a periodic 1--parameter family of Dirac
  operators corresponding to immersions of the disc $M= \{z\in \C\mid
  |z|\leq 1\}$ and denote by $(V_t,\tilde V_t)$ a family of
  self--adjoint, elliptic local boundary conditions. Then its spectral
  flow is
  \[ sf\left(\Dir_{(V_t,\tilde V_t)}\right) = \deg(\tilde V\colon
  S^1\times S^1 \rightarrow S^2), \] where $\tilde V$ denotes the map
  $\tilde V \colon S^1\times S^1=S^1\times \partial M\rightarrow S^2$
  defined by $(t,z)\mapsto \tilde V_t(z)$ (and $S^2$ is outward
  oriented).
\end{Thm}

It is worth mentioning that, in contrast to
Theorems~\ref{Thm:elliptic}, \ref{Thm:selfadjoint}, and
\ref{Thm:index} in which only $V$ plays a role and $\tilde V$ is
complete arbitrary, the spectral flow is determined by $\tilde V$
alone.

\begin{proof}
  By homotopy invariance and additivity of the spectral flow under
  concatenation of path, it is sufficient to prove the claim for one
  example such that $\tilde V$ has degree one. The example used to
  prove the theorem is further discussed in
  Section~\ref{sec:dirac-spheres}, see
  Figure~\ref{fig:conjectured-flow}.

  Denote by $f\colon M\rightarrow S^2$ the conformal immersion of
  $M=\{ z\in \C \mid |z|\leq 1\}$ parametrizing the southern
  hemisphere in the unit sphere $S^2\subset \Im(\H)$ via stereographic
  projection from the north pole $i$ to $jM\subset j \C \subset
  \Im(\H)$.  We choose $\Dir_t=\Dir$, $t\in S^1$ the constant family
  of operators with $\Dir$ the Dirac operator induced \eqref{eq:Dirac}
  by $f$.  Moreover, we take as fixed $V$--field the $B$--field
  $V(z)=B(z)=i$ of $f$.  As 1--parameter family of $\tilde V_t(z)=
  \tilde B_t(z)$--fields we take
  \begin{equation}
    \label{eq:bc-proof}
    \tilde B_t(e^{is}) = -\cos(t/2) i + \sin(t/2) \cos(s) j -
    \sin(t/2) \sin(s) k.
  \end{equation}
  This family of $\tilde B_t$--fields is not $2\pi$--periodic in $t$,
  but it becomes $2\pi$--periods when we compose with the following
  rotation in the $ij$--plane (written with respect to the basis
  $i,j,k$)
  \[
  \begin{pmatrix}
    \cos(t/2) & -\sin(t/2) & 0\\
    \sin(t/2) & \cos(t/2) & 0 \\
    0 & 0 & 1
  \end{pmatrix},
  \]
  see Figure~\ref{fig:spheres}.  After composing with the given
  rotations, the $\tilde B_t$--fields yield a continuous map from
  $S^1\times S^1=\R/2\pi\times \partial M$ to $S^2$ which has degree
  one (since $S^2$ is outward oriented).

 \begin{figure}\label{fig:spheres}
  	\centering

    \includegraphics[width=4cm]{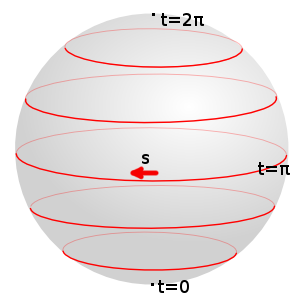}\qquad 
\includegraphics[angle=0,width=4cm]{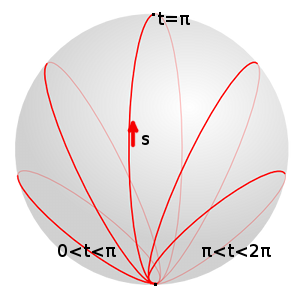}
\caption{(Left) Images of the $\tilde B_t$--fields \eqref{eq:bc-proof}
  and (Right) images of the rotated $\tilde B_t$--fields (with the
  $i$--axis pointing upwards and the $j$--axis pointing to the
  right). The image of the $\tilde B_t$-field at $t=0$ is the south
  pole, at $t=\pi$ the equator parametrized clockwise (!)  in the
  $jk$--plane, and at $t=2\pi$ the north pole. All of the rotated
  $\tilde B_t$--fields go through the south pole at $s=0$; for small
  $t$ its images are circles to the left of the south pole, for
  $t=\pi$ its image is the great circle through the north--pole, and
  for larger $t$ its images are circles right of the south pole.}
  \end{figure}

  Since the spectral flow does not change under rotation (in fact, if
  $\lambda$ is an eigenspinor of $\Dir$ satisfying the $(B,\tilde
  B_t)$ boundary condition, then $\lambda e^{-kt/4}$ is an eigenspinor
  for the rotated boundary condition, because the rotated $\tilde
  B_t$--field is $e^{kt/4} \tilde B_t e^{-kt/4}$), it is sufficient to
  compute the spectral flow for $\tilde B_t$.  Because the spectral
  flow of $\Dir_{(B,\tilde B_t)}$, $t\in [0,2\pi]$ coincides with the
  spectral flow of a closed loop of operators, although originally
  defined as the flow of spectrum through the eigenvalue $\mu=0$, it
  coincides with the flow of spectrum through any other eigenvalue
  $\mu\in \R$. We determine the spectral flow by computing the flow
  through $\mu =-1$.

  The crucial information for the computation of the spectral flow of
  $\Dir_{(B,\tilde B_t)}$, $t\in [0,2\pi]$ is provided by
  Lemma~\ref{Lem:minimal} which shows that the only parameter
  compatible with the eigenvalue $\mu=-1$ is $t=\pi$ for which the
  eigenspace is one dimensional.

  To complete the proof of the theorem it therefore remains to
  compute the sign of the spectral flow. This can be done by the
  following infinitesimal argument: by ``continuity of a finite system
  of eigenvalues'', cf.\ \cite[IV,3.5]{K}, for $t$ in a small
  neighborhood of $t=\pi$ it is clear that $\Dir_{(B,\tilde B_t)}$ has
  a unique eigenvalue close to $-1$ which is necessarily simple.

  By Lemma~\ref{Lem:rotational-symmetric}, integration of $d\tilde f =
  \bar \lambda df\lambda$ for $\lambda$ a simple eigenspinor yields a
  rotational symmetric immersion $\tilde f$. Moreover, by
  \eqref{eq:mc-density-trafo} the mean curvature $\tilde H$ of $\tilde
  f$ is strictly positive for eigenvalues $\mu>-1$ and strictly
  negative for $\mu<-1$. In the case that $t$ is slightly larger that
  $\pi$, the $\tilde B_t$--field is pointing upwards and the immersed
  disc corresponding to the eigenvalue closest to $\mu=-1$ attains its
  minimal height (coordinate in $i$--direction) in the interior of the
  disc. At a point where the minimal height is attained, the tangent
  plane to the immersion is horizontal and therefore, since the
  surface normal is downwards, the mean curvature has to be
  positive. This shows that for $t>\pi$ the eigenvalue closest to
  $\mu=-1$ is larger than $-1$. The sign of the spectral flow across
  $\mu=-1$ is thus positive.
\end{proof}

\begin{Lem}\label{Lem:minimal}
  In the family of Dirac operators $\Dir_{(B,\tilde B_t)}$, $t\in
  [0,2\pi]$ with boundary condition \eqref{eq:bc-proof} as in the
  proof of Theorem~\ref{Thm:spectralflow}, only $t=\pi$ has $\mu=-1$
  as an eigenvalue. The real dimension of the corresponding eigenspace
  is one and its non--trivial elements are nowhere vanishing.
\end{Lem}
\begin{proof}
  By \eqref{eq:mc-density-trafo}, if $\Dir_{(B,\tilde B_t)}$ has an
  eigenspinor $\lambda$ with eigenvalue $\mu=-1$, the immersion
  obtained by integrating $d\tilde f = \bar \lambda df \lambda$ is a
  minimal surface (because $f$ has mean curvature $H=1$). 

  We first prove that this can only happen in the case that $t=\pi$
  for which $\tilde B_t$ takes values in a plane.  To see this, note
  that the $i$--component $f_1$ of the minimal surface $\tilde f$ in
  $\R^3=\Im(\H)$ is a harmonic function so that by Stokes Theorem
  \[ \int_M df_1\wedge * df_1 = \int_{\partial M} f_1 *df_1, \] where
  $*$ denotes the complex structure on $T^*M$ (i.e., minus the usual
  Hodge operator). Hence
  \[ \int_M |df_1|^2 dA = \int_{\partial M} f_1 \tfrac{\partial
    f_1}{\partial \nu} ds \] with $\frac{\partial }{\partial \nu}$ the
  (outer) normal derivative of $f_1$ along the boundary. Now, up to
  some positive scale $\tfrac{\partial f_1}{\partial \nu}$ equals the
  $i$--part of $\tilde B_t$ which is constant and negative for $t<\pi$
  while positive for $t>\pi$.  Because by choosing a suitable
  integration constant we can arrange that $f_1$ has arbitrary
  constant sign, we obtain a contradiction unless $t=\pi$ and $f_1$ is
  constant.  This shows that $\mu=-1$ can only be eigenvalue of
  $\Dir_{(B,\tilde B_t)}$ if $t=\pi$.

  A (possibly branched) immersion $\tilde f$ obtained from an
  eigenspinor for $\mu=-1$ is thus planar and hence of the form $j$
  times a complex holomorphic function.  Up to translation and scaling
  it coincides with $\tilde f(z)=jz$ (because the boundary condition
  implies that the derivatives of $\tilde f$ and $jz$ along the
  boundary differ only by a real factor which is necessarily
  constant).  In particular, the corresponding spinor is unique up to
  a real scale and has no zeros.

  Conversely, one can check that the eigenspace for $t=\pi$ and
  $\mu=-1$ is non--empty (and hence real 1--dimensional) by writing
  down the spinor $\lambda$ which transforms the southern half--sphere
  $f$ to the standard embedding $\tilde f(z)=jz$ of $M$ into $jM$.
\end{proof}

\begin{Lem}\label{Lem:rotational-symmetric}
  Let $\Dir$ be the Dirac operator corresponding to a rotational
  symmetric immersion $f$ of the disc $M=\{z\in \C\mid |z|\leq 1\}$
  and let $(B,\tilde B)$ be a rotational symmetric boundary condition.
  If $\Dir_{(B,\tilde B)}$ has a real 1--dimensional eigenspace, a
  corresponding eigenspinor $\lambda$ is rotational symmetric and
  nowhere vanishing. Integrating $d\tilde f = \bar\lambda df \lambda$
  then yields a rotational symmetric immersion $\tilde f$.
\end{Lem}
\begin{proof}
  Assume the rotation $r_\theta$ by $\theta\in \R$ of the disc $M$
  acts on $\R^3=\Im(R)$ by rotation \[x\in \Im(\H) \mapsto
  e^{-i\theta/2}\, x\, e^{i\theta/2}\in \Im(\H)\] and $f(r_\theta
  (z))= e^{-i\theta/2}f(z) e^{i\theta/2}$ as well as $\tilde
  B(r_\theta (z))= e^{-i\theta/2}\tilde B(z) e^{i\theta/2}$ for all
  $z\in M$ and $\theta\in \R$. Then the map
  \begin{equation}
    \label{eq:representation}
    \lambda \in C^\infty(M,\H) \mapsto (z\mapsto
    e^{i\theta/2}\lambda(r_\theta(z)) e^{-i\theta/2})\in
    C^\infty(M,\H)
  \end{equation}
  induces a $S^1$--representation on the eigenspaces of
  $\Dir_{(B,\tilde B)}$.  A spinor $\lambda$ in a real 1--dimensional
  eigenspace of $\Dir_{(B,\tilde B)}$ is thus rotational symmetric,
  because a real 1--dimensional representation of $S^1$ is trivial.
  Because the zeros of an eigenspinor of $\Dir_{(B,\tilde B)}$ are
  isolated, the rotational symmetric spinor $\lambda$ could only
  vanish on the axis of rotation. But the asymptotics of eigenspinors
  at zeros of positive order, see Lemma~3.9 of \cite{FLPP}, would not
  be compatible with the rotational symmetry of $\lambda$. (More
  generally, a similar argument yields that a spinor that is a higher
  Fourier mode as in \eqref{eq:canonical-basis} has to vanish to order
  $l$ and therefore gives rise to a branched immersion with a branch
  point of order $2l$.)
\end{proof}

\section{Spectral flow and Dirac spectrum of the round 2--sphere
  $S^2$}\label{sec:dirac-spheres}

In the last section of the paper we sketch a relation between the
Dirac spectrum of the round 2--sphere and the example of spectral flow
discussed in the previous section. In particular, the Dirac spectrum
of the round 2--sphere is visualized in
Figure~\ref{fig:periodic-table-of-dirac} and spectral flow is
visualized in Figures~\ref{fig:conjectured-flow} and
\ref{fig:flow-to-snowman}.

As in the previous section, we denote by $\Dir_{(B,\tilde B_t)}$ the
Dirac operator induced by the immersion $f$ of the unit disc into
the southern hemisphere with rotationally symmetric boundary
conditions \eqref{eq:bc-proof}.  The following theorem links its
spectrum for $t\in 2\pi \Z$ (when the $\tilde B_t$--field of the
boundary condition is constant and vertical) to the Dirac spectrum of
the full round sphere.

\begin{Thm}\label{Thm:spheres} 
  For $t\in 2\pi \Z$, a branched immersion $\tilde f$ with $d\tilde
  f=\bar \lambda df\lambda$ for $\lambda$ an eigenspinor of
  $\Dir_{(B,\tilde B_t)}$ orthogonally intersects the $jk$--plane
  along the boundary curve.  Reflection in the $jk$--plane extends
  $\tilde f$ to a smooth branched immersion of the sphere. The
  corresponding extension of the spinor $\lambda$ is an eigenspinor of
  the Dirac operator for the round 2--sphere.
\end{Thm}

\begin{proof}
  That the boundary curve of $\tilde f$ is contained in a plane
  parallel to the $jk$--plane is clear, because its tangent vector is
  perpendicular to $\tilde B_t=\pm i$. Similarly, $\tilde f$
  intersects this translate of the $jk$--plane orthogonally, because
  the normal of $\tilde f$ is as well perpendicular to $\tilde B_t$.
  Reflection of $\tilde f$ yields a $C^2$--immersion of the
  sphere. (That the immersion of the full sphere is $C^1$ is obvious
  from the construction.  That the second derivative of the full
  sphere is continuous along the boundary curve of the reflected
  half--spheres can be proven using that the latter, being contained
  in the plane of reflection, is a curvature line, so that, with
  respect to polar coordinates on the parameter domain, the second
  fundamental form is diagonal along the symmetry axis.)  The
  corresponding extension of the spinor $\lambda$ is thus a
  $C^1$--eigenspinor of the Dirac operator $\Dir$ for the full round
  sphere. By elliptic regularity it is smooth and so is the
  corresponding immersion $\tilde f$ of the sphere.
\end{proof}

The branched immersions of the sphere thus arising from our Dirac
boundary problem for vertical $\tilde B_t$ are examples of surfaces
known as Dirac spheres \cite{R}: a \emph{Dirac sphere} is an immersion
$\tilde f$ of the sphere obtained via $d\tilde f =\bar\lambda
df\lambda$ from an eigenspinor $\lambda$ of the Dirac operator $\Dir$
induced \eqref{eq:Dirac} by the round immersion $f$ of $S^2$ (with
$f\colon \C\cup \{\infty\} \rightarrow S^2$ denoting the extension of
the above parametrization of the southern hemisphere). Dirac spheres
are special examples of soliton spheres \cite{T1,BP} related to
soliton solutions of the mKdV--equation.

Theorem~\ref{Thm:spheres} allows to compute the spectrum of
$\Dir_{(B,\tilde B_t)}$ with vertical boundary condition $t\in 2\pi\Z$
from the spectrum of $\Dir$ for the full unit sphere $S^2$.  Up to a
real shift, $\Dir$ for $S^2$ coincides with the usual spin Dirac
operator of the round sphere.  Its spectrum is $\{\mu \in \Z \mid
\mu\neq -1 \}$ with eigenspaces of quaternionic dimension $|\mu+1|$,
see e.g.\ \cite{R} or \cite{S} for references to the physical
literature.

\begin{figure}\label{fig:periodic-table-of-dirac}
  	\centering

    \includegraphics[width=15cm]{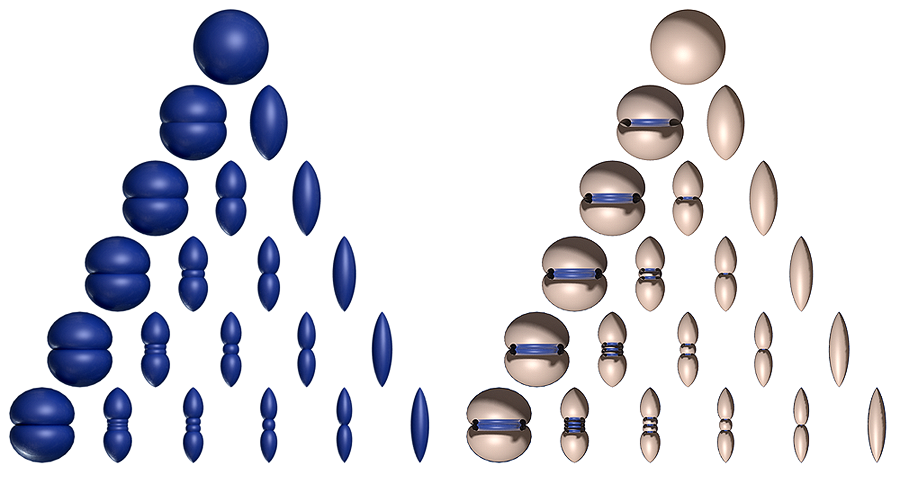}\qquad 
    \caption{The periodic table of Dirac spheres. Vertical rows show
      the Dirac spheres corresponding to the canonical basis
      \eqref{eq:canonical-basis} for eigenvalues $\mu=0$, $1$,..., $5$
      or, equivalently, $\mu=-2$, $-3$,...,$-7$ (the images for the
      eigenvalues $\mu$ and $-(2+\mu)$ are the same, as the
      corresponding branched immersions are geometrically related by
      point reflections). The leftmost surfaces are immersed surfaces
      of revolution. The other surfaces rotate with higher frequencies
      and are branched on the axis of rotation. The frequency raises
      by two for each step to the right. (Pictures by Keenan Crane)}
\end{figure}

The spectrum of $\Dir$ for $S^2$ can be visualized by the ``periodic
table of Dirac spheres'', Figure~\ref{fig:periodic-table-of-dirac},
obtained by Fourier decomposing the quaternionic eigenspace of $\Dir$
for a given $\mu$ with respect to the $S^1$--representation
\eqref{eq:representation}. This yields a quaternionic basis
$\lambda_l$, $l=0$,..., $|\mu+1|-1$ satisfying
\begin{equation}
  \label{eq:canonical-basis}
  \lambda_l(r_\theta(z)) = e^{-i \theta/2} \lambda_l(z) e^{i (1/2+l)
    \theta}.  
\end{equation}
One can check that the branched immersions $\tilde f_l$, $l=0$,...,
$|\mu+1|-1$ obtained from this standard basis via $d\tilde f_l =
\bar\lambda_l df\lambda_l$ are symmetric with respect to the reflection at
a translate of the $jk$--plane.  Moreover, if $\mu\geq 0$ the
restriction of the spinors $\lambda_l$ to the unit disc yields a
(complex) basis of the $\mu$--eigenspace for the boundary condition
$\tilde B=i$, while for $\mu\leq -2$ one obtains a basis of the
$\mu$--eigenspace for the boundary condition $\tilde B=-i$. (Left
multiplication by the quaternion $j$ yields a basis for the opposite
vertical boundary condition.)

From Section~\ref{sec:spectral-flow} and the preceding discussion we
know the eigenspaces of $\Dir_{(B,\tilde B_t)}$ belonging to the
eigenvalues $\mu=-2$, $-1$, and $0$ for $t=0$, $\pi$, and $2\pi$,
respectively. The corresponding immersed discs are shown in
Figure~\ref{fig:known}. In contrast to $t=\pi$, for which the
$-1$--eigenspace is real 1--dimensional, if $t=0$ or $2\pi$ the
$\tilde B_t$--field is constant and all rotations of the corresponding
immersed disc can be obtained from eigenspinors of $\Dir_{(B,\tilde
  B_t)}$. This reflects the fact that the eigenspaces for the
corresponding eigenvalues $\mu=-2$ and $0$ are then (real)
2--dimensional.

\begin{figure}\label{fig:known}
  	\centering

        \includegraphics[width=12cm]{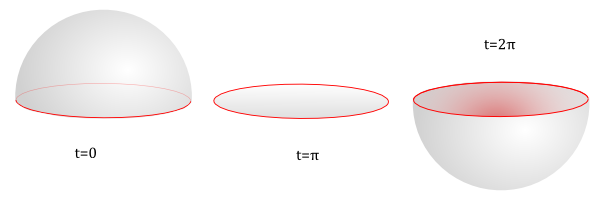}\qquad
        \caption{The northern hemisphere, the planar disc, and the
          original immersion $f$, all with downward orientation,
          correspond to eigenspinors of $\Dir_{(B,\tilde B_t)}$ for
          $t=0$, $\pi$, and $2\pi$ and eigenvalue $-2$, $-1$, and
          $0$, respectively.}
\end{figure}

The principle of ``continuity of a finite systems of eigenvalues''
\cite[IV,3.5]{K} suggests that the spectral flow along the path
$\Dir_{(B,\tilde B_t)}$, $t\in[0,2\pi]$ (which is $1$, as shown in the
proof of Theorem~\ref{Thm:spectralflow}) can be geometrically realized
as an ``interpolation'' between the surfaces shown in
Figure~\ref{fig:known}.  Indeed, taking the largest eigenvalue
$\mu(t)\leq-1$ of $\Dir_{(B,\tilde B_t)}$ for every $t\in [0,\pi]$ and
the smallest eigenvalue $\mu(t)\geq-1$ of $\Dir_{(B,\tilde B_t)}$ for
every $t\in [\pi,2\pi]$ yields a continuous function $\mu\colon
[0,2\pi] \rightarrow \R$.  One would expect that, corresponding to
$\mu(t)$, $t\in [0,2\pi]$, there is a family of immersed discs similar
to Figure~\ref{fig:conjectured-flow} which can be obtained from
eigenspinors of $\Dir_{(B,\tilde B_t)}$ for $\mu(t)$. To make this
rigorous one has to control whether one can continuously deform
eigenspinors through possible bifurcation points $t$ for which
$\mu(t)$ is not a simple eigenvalue of $\Dir_{(B,\tilde B_t)}$. As
indicated in Figure~\ref{fig:conjectured-flow}, it should be possible
to realize spectral flow by a deformation through surfaces of
revolution. This can be checked by spectral theory for 1--dimensional
Dirac operators.

 \begin{figure}\label{fig:conjectured-flow}
  	\centering

    \includegraphics[width=15cm]{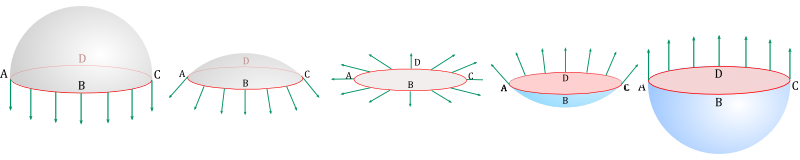}\qquad 
    \caption{Expected qualitative behavior of immersions corresponding
      to eigenspinors realizing the spectral flow occurring in the
      proof of Theorem~\ref{Thm:spectralflow} (as $t$ goes from $0$ to
      $2\pi$ and $\mu$ from $-2$ to $0$).}
\end{figure}

A similar analysis should show that the spectral flow of
$\Dir_{(B,\tilde B_t)}$ for $t$ between $2\pi k$ and $2\pi (k+1)$,
$k\in \Z$, and $\mu$ between $l \in \Z\backslash\{-1, -2\}$ and $l +1$
can be geometrically realized by a deformation through surfaces of
revolution obtained from eigenspinors of $\Dir_{(B,\tilde B_t)}$. See
e.g.\ Figure~\ref{fig:flow-to-snowman} for the expected deformation
between $\mu=-2$ and $\mu=-3$.

 \begin{figure}\label{fig:flow-to-snowman}
  	\centering

    \includegraphics[width=12cm]{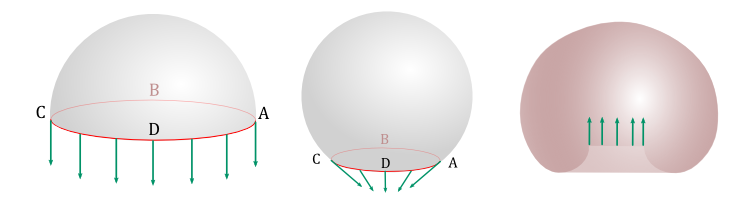}\qquad 
    \caption{Expected qualitative behavior of immersions corresponding
      to eigenspinors realizing the spectral flow for the boundary
      condition \eqref{eq:bc-proof} as $t$ goes from $0$ to $-2\pi$
      and $\mu$ from $-2$ to $-3$.}
\end{figure}

\textbf{Summary:} We conclude the paper by summarizing the preceding
discussion of the family $\Dir_{(B,\tilde B_t)}$ of Dirac operators
induced by the round half--sphere with periodic bi--normal boundary
conditions.  The arising picture is the following:
\begin{itemize}
\item for vertical boundary conditions $t\in 2\pi \Z$ the spectrum of
  $\Dir_{(B,\tilde B_t)}$ coincides with the spectrum of the full
  round sphere (Theorem~\ref{Thm:spheres}) and
\item the spectral flow under a half rotation $t\leadsto t+2\pi$ of
  the bi--normal field is one (Theorem~\ref{Thm:spectralflow}).
\end{itemize}
This means that, upon a half rotation of the bi--normal field from $t=
2\pi k$ to $2\pi (k+1)$, $k\in \Z$, all eigenvalues flow upwards with
multiplicity one (otherwise said, after a half rotation the spectrum
itself is unchanged, but when viewing the eigenvalues taken with their
multiplicities as an ordered sequence of real numbers that
continuously depend on $t$, under $t\leadsto t+2\pi$ all eigenvalues
move up one step in this ordered sequence of eigenvalues).

One would expect that one can follow this upward flow of eigenvalues
by a continuous family of eigenspinors for the respective
eigenvalues. This would allow to realize spectral flow geometrically
by immersed discs, presumably with rotational symmetry. Assuming this
is possible,
\begin{itemize}
\item the upward flow of eigenvalues in the negative part of the
  spectrum $\mu<-2$ would amount to ``continuously'' climbing up on the
  leftmost side in the periodic table of Dirac spheres in
  Figure~\ref{fig:periodic-table-of-dirac} (for the expected
  ``backward'' flow $t\leadsto t-2\pi$ from $-2$ to $-3$, see
  Figure~\ref{fig:flow-to-snowman}),
\item the flow from $\mu=-2$ to $\mu=0$ would ``continuously'' interpolate
  between half--spheres with opposite orientations (see
  Figure~\ref{fig:conjectured-flow}), and
\item the upward flow in the positive part $\mu\geq 0$ would amount to
  going down ``continuously'' on the leftmost side in the periodic table
  of Dirac spheres in Figure~\ref{fig:periodic-table-of-dirac}.
\end{itemize}
This indicates that geometrically the spectral flow in our example
corresponds to ``wrapping up'' the disc/half--sphere by rotating its
bi--normal field around the boundary curve. From this perspective
spectral flow appears as a continuous geometric realization of the
process of ``adding solitons''.

\end{document}